\documentclass[12pt]{amsart}
\usepackage{etex}
\usepackage{etoolbox}
\patchcmd{\thebibliography}{*}{}{}{}
\usepackage{amssymb}
\usepackage{cite}
\usepackage{booktabs}
\usepackage{url}
\usepackage{hyphenat}
\usepackage{mathtools}
\usepackage{pifont}
\usepackage[all,cmtip]{xy}
\usepackage{ifpdf}
\usepackage{enumitem}

\usepackage{xcolor}
\usepackage{graphicx}
\usepackage[lmargin=1in,rmargin=1in,tmargin=1in,bmargin=1in]{geometry}
\usepackage{subfig}
\usepackage{mathrsfs}
\usepackage{overpic}
\usepackage{xr-hyper}

\definecolor{darkblue}{rgb}{0,0,0.4} 
\definecolor{darkgreen}{rgb}{0,0.4,0} 
\definecolor{mypurple}{rgb}{0.4,0.18,.57} 
\usepackage[colorlinks=true, citecolor=darkblue, filecolor=darkblue, linkcolor=darkblue,urlcolor=darkblue]{hyperref} 
\usepackage[all]{hypcap}

\usepackage{tikz}
\usetikzlibrary{matrix,arrows,3d}
\tikzset{zxplane/.style={canvas is zx plane at y=#1,very thin}}
\tikzset{xyplane/.style={canvas is xy plane at z=#1,very thin}}
\tikzset{yzplane/.style={canvas is yz plane at x=#1,very thin}}

\allowdisplaybreaks

\graphicspath{{draws/}{}}

\newcommand{\RR}{\mathbb R}
\newcommand{\CC}{\mathbb C}

\newcommand{\FF}{\mathbb F}

\newcommand{\co}{\nobreak\mskip2mu\mathpunct{}\nonscript
  \mkern-\thinmuskip{:}\penalty300\mskip6muplus1mu\relax}
\newcommand{\from}{\co}

\newcommand{\bdy}{\partial}
\newcommand{\into}{\hookrightarrow}

\newcommand{\lbracket}{[}
\newcommand{\rbracket}{]}

\renewcommand{\emptyset}{\varnothing}

\theoremstyle{plain}

\numberwithin{equation}{section}
\newtheorem{theorem}[equation]{Theorem}

\newtheorem{proposition}[equation]{Proposition}
\newtheorem{lemma}[equation]{Lemma}
\newtheorem{corollary}[equation]{Corollary}

\theoremstyle{definition}

\theoremstyle{remark}

\newtheorem{remark}[equation]{Remark}

\newcommand{\HF}{\mathit{HF}}

\newcommand{\CF}{{\mathit{CF}}}

\newcommand{\x}{\mathbf x}
\newcommand{\y}{\mathbf y}

\newcommand\HH{\mathit{HH}}

\newcommand\Hochschild\HH

\newcommand\Id{\mathbb{I}}

\newcommand{\Field}{{\FF_2}}

\DeclareMathOperator{\Fix}{Fix}

\newcommand{\ol}[1]{\overline{#1}{}}
\newcommand{\wt}[1]{\widetilde{#1}{}}

\makeatletter
\newcommand\honestalg[3]{\bigl\lbracket
\begin{smallmatrix} #1\@ifempty{#3}{}{&#3} \\ #2 \end{smallmatrix}
\bigr\rbracket}

\makeatother

\newcommand{\JSpace}{\mathcal{J}}

\newcommand{\ECat}{\mathscr{E}}

\newcommand{\Complexes}{\mathsf{Kom}}
\DeclareMathOperator{\hocolim}{hocolim}
\newcommand{\KhSymp}{\mathit{Kh}_{\mathit{symp}}}

\newcommand{\KCSymp}{\mathcal{C}_{\mathit{Kh},\mathit{symp}}}
\newcommand{\eKCSymp}{\mathcal{C}_{\mathit{Kh}}^{\mathit{symp,free}}}

\DeclareMathOperator{\Hilb}{Hilb}
\newcommand{\CipLag}{\mathcal{K}}
\newcommand{\ssspace}[1]{\mathcal{Y}_{#1}}

\newcommand{\ECF}[1][{}]{\widetilde{\CF}^{\raisebox{-4pt}{$\scriptstyle #1$}}}

\makeatletter

\newcommand*\wthelper[2]{%
        \hbox{\dimen@\accentfontxheight#1%
                \accentfontxheight#11.3\dimen@
                $\m@th#1\widetilde{#2}$%
                \accentfontxheight#1\dimen@
        }%
}
\newcommand*\accentfontxheight[1]{%
        \fontdimen5\ifx#1\displaystyle
                \textfont
        \else\ifx#1\textstyle
                \textfont
        \else\ifx#1\scriptstyle
                \scriptfont
        \else
                \scriptscriptfont
        \fi\fi\fi3
}
\makeatother

\begin{document}
\title[Correction to A flexible construction\dots]{Correction to the
  paper ``A flexible construction of equivariant Floer homology and applications''}

\author{Kristen Hendricks}
 \address{Mathematics Department, Rutgers University\\
   New Brunswick, NJ 08901}
   \thanks{KH was supported by NSF grant DMS-1663778 and NSF CAREER grant DMS-1751857.}
\email{\href{mailto:kristen.hendricks@rutgers.edu}{kristen.hendricks@rutgers.edu}}

\author{Robert Lipshitz}
 \address{Department of Mathematics, University of Oregon\\
   Eugene, OR 97403}
\thanks{RL was supported by NSF grant DMS-1149800 (version 1) and
  DMS-1810893 (revisions).}
\email{\href{mailto:lipshitz@uoregon.edu}{lipshitz@uoregon.edu}}

\author{Sucharit Sarkar}
\thanks{SS was supported by NSF grant DMS-1643401}
\address{Department of Mathematics, University of California\\
  Los Angeles, CA 90095}
\email{\href{mailto:sucharit@math.ucla.edu}{sucharit@math.ucla.edu}}

\date{\today}

\begin{abstract}
  We correct a mistake regarding almost complex structures on Hilbert
  schemes of points in surfaces in~\cite{HEquivariant}. The error does
  not affect the main results of the paper, and only affects the
  proofs of invariance of equivariant symplectic Khovanov homology and
  reduced symplectic Khovanov homology. We give an alternate
  proof of the invariance of equivariant symplectic Khovanov homology.
\end{abstract}

\maketitle


\section{The mistake}
At several points in Section 7 of our paper~\cite{HEquivariant} we assert that
an almost complex structure $j$ on the algebraic surface $S$ used to define
symplectic Khovanov homology induces an almost complex structure $\Hilb^n(j)$ on
the Hilbert scheme (or Douady space, following~\cite{Douady66}) $\Hilb^n(S)$ of length $n$ subschemes of $S$. If $j$ is a
complex structure this is true, but there is no known extension of the Hilbert
scheme of points in a complex manifold to the almost-complex case. (Indeed, even
the definition of the Hilbert scheme as a set depends on the complex structure.)
See also Voisin's paper~\cite{Voisin} for some interesting steps in this
direction and further discussion.

This (false) principle is used in a ``cylindrical'' formulation of
symplectic Khovanov homology in~\cite[Lemma 7.10]{HEquivariant}, which is
then used in the proof of stabilization invariance for symplectic Khovanov
homology in~\cite[Section 7.4.1]{HEquivariant}, equivariant symplectic
Khovanov homology in~\cite[Theorem 1.26]{HEquivariant}, and reduced
symplectic Khovanov homology in~\cite[Theorem 7.25]{HEquivariant}. (See
also Abouzaid-Smith's paper~\cite{AbouzaidSmith:arc-alg} for a more careful
cylindrical reformulation of the curves in symplectic Khovanov
homology in certain cases, and Mak-Smith's recent paper~\cite{MakSmith:cylindrical} for a more general cylindrical reformulation.)

Below, we give a corrected, weaker version of the offending Lemma
7.10, and an independent proof of equivariant stabilization invariance
(i.e.,~\cite[Theorem 1.26]{HEquivariant}). We have not been able to
correct the proof of stabilization invariance of reduced symplectic
Khovanov homology (i.e.~\cite[Theorem 7.25]{HEquivariant}). So, reduced
symplectic Khovanov homology is invariant under isotopies and
handleslides, but is only conjectured to be invariant under
stabilization.

\emph{Acknowledgments.} We thank Tomohiro Asano for pointing out our mistake,
Nick Addington for helpful conversations, and the referee for further
corrections and suggestions.

\section{The corrected lemma}
Recall that the curves we consider in the cylindrical formulation of
symplectic Khovanov homology are maps
\[
(7.9)\qquad\!\!
  \psi\co(X,\bdy X)\to \bigl(\RR\times[0,1]\times S,(\RR\times\{0\}\times (\Sigma'_{A_1}\cup\dots\cup \Sigma'_{A_n}))\cup(\RR\times\{1\}\times (\Sigma_{B_1}\cup\dots\cup \Sigma_{B_n}))\bigr),
\]
where $X$ is a Riemann surface with boundary and $2n$ boundary
punctures, $\psi$ is asymptotic to $\{-\infty\}\times[0,1]\times\x$
and $\{+\infty\}\times[0,1]\times \y$, and
$\pi_{\RR\times[0,1]}\circ\psi$ is an $n$-fold branched covering.

The incorrect Lemma 7.10 introduced the following condition for these maps:
\begin{enumerate}[label=(YC)]
\item\label{item:YC} \ \ \  The map $
  (\Id \times\Id \times i)\circ \psi \co (X\setminus B(\psi))\to \RR\times[0,1]\times\CC
  $
  is an embedding.
\end{enumerate}

The following is a corrected version of Lemma 7.10:

\noindent\textbf{Lemma 7.10${}^{\boldsymbol{\prime}}$} {\em
  The set of holomorphic disks in the complex manifold $\ssspace{n}$ connecting
  $\x$ to $\y$ and which are transverse to the big diagonal $\Delta$ is in
  bijection with the set of holomorphic maps as in Formula (7.9)
  satisfying condition (YC).}

\begin{remark}
  To emphasize, the corrected Lemma 7.10${}^{\prime}$ is about a
  complex structure on $S$ and the induced complex structure on
  $\ssspace{n}$, not an almost complex structure on $S$ or
  $\ssspace{n}$. In particular, there is no assertion that the moduli
  spaces in Lemma 7.10${}^{\prime}$ are transversely cut out.  The
  corrected Lemma 7.10${}^{\prime}$ is not used in the rest of this
  note.
\end{remark}

\section{Corrected proof of equivariant stabilization invariance}
\subsection{Background}\label{sec:background}
\subsubsection{Skein triangles}
The corrected proof of stabilization invariance is similar to our
proof of stabilization invariance for the equivariant Heegaard Floer
homology of branched double covers from~\cite[Theorem
1.24]{HEquivariant}. Here are the analogues for symplectic Khovanov
homology of the results about Heegaard Floer homology that proof used:

\begin{theorem}\cite{SeidelSmith6:Kh-symp}\cite{Waldron:KhSympMaps}\label{thm:SS-invt}
  Symplectic Khovanov homology is a link invariant.
\end{theorem}

\begin{figure}
  \centering
  \includegraphics[scale=.75]{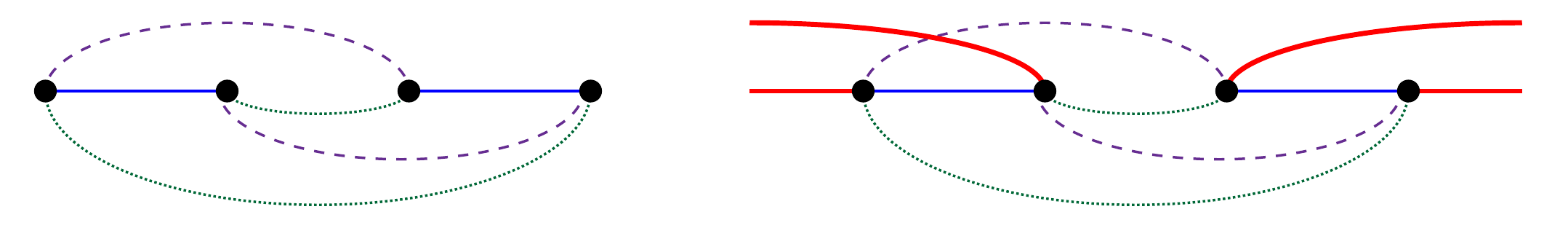}
  \caption{\textbf{Skein triangle via bridges.} Left: the bridges $B$ (\textcolor{blue}{solid}), $B'$ (\textcolor{darkgreen}{dotted}), and $B''$ (\textcolor{mypurple}{dashed}). Right: the same picture, but with $A$-arcs (\textcolor{red}{thick}) to illustrate why $(B,B',B'')$ correspond to a skein relation.}
  \label{fig:skein-bridges}
\end{figure}

\begin{theorem}\label{thm:skein-tri}\cite{AbouzaidSmith:KhSympKh}
  Let $B$, $B'$, and $B''$ be collections of bridges which differ as
  shown in Figure~\ref{fig:skein-bridges} in part of the diagram and
  are small isotopic copies of each other in the rest of the
  diagram. For any collection
  of bridges $A$ there is an exact triangle
  \[
    \xymatrix{
      & \HF(\CipLag_A,\CipLag_B)\ar[dr] & \\
      \HF(\CipLag_A,\CipLag_{B''})\ar[ur] & & \HF(\CipLag_A,\CipLag_{B'})\ar[ll]
    }
  \]
  Further, there are elements
  $\alpha\in\CF(\CipLag_B,\CipLag_{B'})$,
  $\beta\in\CF(\CipLag_{B'},\CipLag_{B''})$, and
  $\gamma\in\CF(\CipLag_{B''},\CipLag_{B})$ so that the maps in the
  exact triangle come from counting holomorphic triangles on
  $(\CipLag_A,\CipLag_B,\CipLag_{B'})$,
  $(\CipLag_A,\CipLag_{B'},\CipLag_{B''})$, and
  $(\CipLag_A,\CipLag_{B''},\CipLag_{B})$ with one corner at $\alpha$,
  $\beta$, and $\gamma$, respectively.
\end{theorem}
\begin{proof}
  Using ideas of Abouzaid-Ganatra, Abouzaid-Smith show that there is
  an exact triangle of bimodules over the Fukaya category of
  $\ssspace{n}$ relating the identity bimodule $\Id$, the composition
  of a cup and a cap, and the half-twist $\tau$; see~\cite[Proposition
  7.4]{AbouzaidSmith:KhSympKh}. Evaluating these three bimodules on
  $\CipLag_B$ as the first object gives three one-sided
  modules---$\CipLag_B$, a module equivalent to $\CipLag_{B'}$, and a
  module equivalent to $\CipLag_{B''}$. Since equivalences preserve
  exact triangles, this implies that there is an exact triangle
  relating $\CipLag_B$, $\CipLag_{B'}$, and $\CipLag_{B''}$. In
  particular, the maps $\CipLag_B\to\CipLag_{B'}$,
  $\CipLag_{B'}\to\CipLag_{B''}$, and $\CipLag_{B''}\to\CipLag_B$ come
  from elements $\alpha\in\CF(\CipLag_B,\CipLag_{B'})$,
  $\beta\in\CF(\CipLag_{B'},\CipLag_{B''})$, and
  $\gamma\in\CF(\CipLag_{B''},\CipLag_{B})$, so that for any other
  Lagrangian $L$, counting holomorphic triangles on
  $(L,\CipLag_B,\CipLag_{B'})$ with a corner at $\alpha$ (respectively
  on $(L,\CipLag_{B'},\CipLag_{B''})$ with a corner at $\beta$, on
  $(L,\CipLag_{B''},\CipLag_{B})$ with a corner at $\gamma$) gives an
  exact triangle relating $\HF(L,\CipLag_B)$, $\HF(L,\CipLag_{B'})$,
  and $\HF(L,\CipLag_{B''})$. Taking $L=\CipLag_A$ gives the result.
\end{proof}

\begin{remark}
  For appropriate choices of diagrams, all of the
  generators of $\CF(\CipLag_B,\CipLag_{B'})$,
  $\CF(\CipLag_{B'},\CipLag_{B''})$, and
  $\CF(\CipLag_{B''},\CipLag_{B})$ are fixed by the $O(2)$-action.
\end{remark}

\subsubsection{Projected domains}
While the arguments below take place in the Hilbert scheme, and avoid
a cylindrical formulation of symplectic Khovanov homology, we will
make use of the concept and some properties of projected domains
from~\cite[Section 7.1.3]{HEquivariant}. Fix a bridge diagram
$(A=\{A_i\},B=\{B_i\})$ for a link $L$. Choose a point $z_i$ in each
connected component of $\CC\setminus (A\cup B)$. Each $z_i$ gives a
subvariety $\pi^{-1}(z_i)\times\Hilb^{n-1}(S)$ of $\Hilb^n(S)$
consisting of those length-$n$ subschemes where at least one point
lies over $z_i$. Fix a neighborhood $U_i$ of
$\pi^{-1}(z_i)\times\Hilb^{n-1}(S)$, small enough that the closure of
$U_i$ is disjoint from $\CipLag_A\cup\CipLag_B$.

Let $R$ denote the non-compact region of
$\CC\setminus(A\cup B)$.  Then $\pi^{-1}(R)\times\Hilb^{n-1}(S)$, the
set of length-$n$ subschemes where at least one point lies over $R$,
is an open subset of $\Hilb^n(S)$. Order the $z_i$ above so that $z_0\in R$.

We will only consider almost complex structures $J$ on $\ssspace{n}$,
compatible with the symplectic form described in
Section~\ref{sec:convex} below, with the following three additional
properties:
\begin{enumerate}[label=(J-\arabic*)]
\item\label{item:J-first} The almost complex structure $J$ agrees with the standard
  complex structure $\Hilb^n(j)$ on $U_i$. In particular, each
  $(\pi^{-1}(z_i)\times\Hilb^{n-1}(S))\cap\ssspace{n}$ is a
  $J$-holomorphic submanifold.
\item The almost complex structure $J$ agrees with the standard
  complex structure $\Hilb^n(j)$ on $\pi^{-1}(\overline{R})\times\Hilb^{n-1}(S)$.
\item\label{item:J-convex}\label{item:J-last} The almost complex structure $J$ agrees
  with the standard complex structure $\Hilb^n(j)$ outside a compact
  subset. (This is not implied by the previous restriction because the
  fibers of $\pi$ are themselves non-compact.)
\end{enumerate}
(Compare~\cite[Definition 3.1]{OS04:HolomorphicDisks}.)

Since $(\pi^{-1}(z_i)\times\Hilb^{n-1}(S))\cap\ssspace{n}$ is proper
and disjoint from $\CipLag_A\cup\CipLag_B$, each
$(\pi^{-1}(z_i)\times\Hilb^{n-1}(S))\cap\ssspace{n}$ is dual to a
relative cohomology class
$PD[\pi^{-1}(z_i)\times\Hilb^{n-1}(S)]\in
H^2(\ssspace{n},\CipLag_A\cup\CipLag_B)$.  Given a Whitney disk $u$
for $(\CipLag_A,\CipLag_B)$, let $n_{z_i}(u)$ be the result of
evaluating the cohomology class
$PD[\pi^{-1}(z_i)\times\Hilb^{n-1}(S)]$ on $[u]$. The tuple
$(n_{z_i}(u))$ is the \emph{projected domain} of $u$. Sometimes, we
think of the projected domain as an element of
$H_2(\CC\cup\{\infty\},A\cup B)$, where $n_{z_i}(u)$ is the coefficient
of the region containing $z_i$.

If $u$ is $J$-holomorphic then the projected domain of $u$ has the following properties:
\begin{enumerate}[label=(D-\arabic*)]
\item For each $i$, $n_{z_i}(u)\geq 0$. (Compare~\cite[Lemma
  3.2]{OS04:HolomorphicDisks}.)
\item\label{item:domain-2} The value $n_{z_0}(u)=0$, and in fact $u(D^2)\cap (R\times\Hilb^{n-1}(S))=\emptyset$. 
\item If $n_{z_i}(u)=0$ for each $i$ then $u$ is constant.
\end{enumerate}
The first two statements follow from positivity of intersections of
complex submanifolds. The third follows from the fact that any such
Whitney disk is homotopic to a constant disk, hence has zero area, and
hence is itself constant.

Finally, no non-constant, $J$-holomorphic Whitney disk is
entirely contained in the subspace
\[
  \bigl(\pi^{-1}(\overline{R})\times\Hilb^{n-1}(S)\bigr)\cup \bigcup_i U_i
\]
where we have imposed constraints on $J$. So, the usual transversality
arguments for $J$-holomorphic curves (see, e.g.,~\cite[Chapter
3]{MS04:HolomorphicCurvesSymplecticTopology}) imply generic
transversality for (one parameter families of) almost complex
structures $J$ in this class.

\subsubsection{Symplectic forms and convexity at infinity}\label{sec:convex}

In~\cite{HEquivariant}, we worked in the setting of symplectic
manifolds which are convex at infinity, in the sense
of~\cite{EliashbergGromov91:convex}.  An alternative notion of
convexity comes from~\cite[Section 0.4]{Gromov85}: a symplectic
manifold $M$ is \emph{$I$-convex at infinity} (or simply
\emph{$I$-convex}) if there is an exhaustion of $M$ by open sets
$V_1\subset V_2\subset\cdots$ with $\overline{V_i}$ compact and such that if
$u\co D^2\to M$ is $I$-holomorphic with $u(\bdy D^2)\subset V_i$ then
$u(D^2)\subset V_{i+1}$. For example, the maximum modulus theorem
implies that any affine variety is $I$-convex. Observe also that the
notion of a symplectic manifold being $I$-convex depends only on $I$
near infinity (i.e., outside a compact set). As noted
in~\cite{HLS:Lie}, the arguments from~\cite{HEquivariant} work for symplectic manifolds
which are $I$-convex at infinity for some $G$-invariant almost complex
structure $I$ defined outside a compact set.

In particular, the $I$-convex setting is convenient for
symplectic Khovanov homology. Let $I$ be the complex structure on
$\ssspace{n}$ inherited from $\Hilb^n(S)$. The complex structure $I$
is $O(2)$-invariant and, since $\ssspace{n}$ is an affine variety (see
also~\cite[Theorem 1.2]{Manolescu06:nilpotent}), $\ssspace{n}$ is
$I$-convex. Inspired by Perutz's construction
in~\cite{Perutz07:HamHand}, in~\cite[Lemma
5.5]{AbouzaidSmith:arc-alg}, Abouzaid-Smith construct a K\"ahler form
$\omega'$ on $\Hilb^n(S)$ (with respect to $I$) whose restriction to
$\ssspace{n}$ is exact and agrees with the product form outside a
neighborhood of the diagonal. (In the notation of~\cite[Lemma
5.5]{AbouzaidSmith:arc-alg}, we choose $\omega$ to be an exact
K\"ahler form on $S$.)  We saw in~\cite[Lemma 4.24]{HLS:Lie} that
averaging $\omega'$ gives an $O(2)$-invariant K\"ahler form on
$\ssspace{n}$ (still with respect to $I$) which is still exact and
still agrees with the product symplectic form outside a neighborhood
of the diagonal.

If $V$ is an open subset of $\CC$, then
$\ssspace{n}\cap \Hilb^n(\pi^{-1}(V))$ is also $I$-convex, for the
same reason that point~\ref{item:domain-2}, above, holds. In
particular, complex structures satisfying
condition~\ref{item:J-convex} satisfy the convexity-at-infinity
condition required for the constructions in~\cite{HEquivariant}.

\subsection{General K\"unneth theorem for the freed Floer complex}
Another ingredient in our proof of Proposition~\ref{prop:equi-stab-inv} is
a K\"unneth theorem for equivariant symplectic Khovanov
homology. We give a general K\"unneth theorem for the equivariant
Floer complex in this section and specialize to the case of symplectic
Khovanov homology in the next section.

\begin{remark}
  Throughout this section, the convexity assumption
  from~\cite[Hypothesis 3.2]{HEquivariant} can be replaced with
  $I$-convexity for some $G$-invariant almost complex structure $I$
  defined outside a compact set. See Section~\ref{sec:convex} for
  further discussion and references.
\end{remark}

\begin{theorem}\label{thm:external-Kunneth}
  Suppose that $H$ acts on $(M,L_0,L_1)$ and $H'$ acts on
  $(M',L'_0,L'_1)$, both satisfying~\cite[Hypothesis 3.2]{HEquivariant}.
  Then the action of $H\times H'$ on $(M\times M',L_0\times L'_0,L_1\times L'_1)$
  satisfies~\cite[Hypothesis 3.2]{HEquivariant} and there is a quasi-isomorphism
  \[
    \ECF[H\times H'](L_0\times L'_0,L_1\times L'_1)\simeq
    \ECF[H](L_0,L_1)\otimes_{\Field} \ECF[H'](L'_0,L'_1)
  \]
  of chain complexes over $\Field[H\times H']$.
\end{theorem}
\begin{proof}
  To keep notation simple, we will prove the result in the case that
  $L_0\pitchfork L_1$ and $L'_0\pitchfork L'_1$; the extension to
  non-transverse intersections is the same as~\cite[Section
  3.6]{HEquivariant}.
  
  Observe that if $\wt{J}$ (respectively $\wt{J}'$) is an eventually cylindrical
  almost complex structure on $M$ (respectively $M'$) so that the moduli spaces
  of $\wt{J}$-holomorphic Whitney disks in $M$ (respectively
  $\wt{J}'$-holomorphic Whitney disks in $M'$) are transversely cut out then the
  moduli space of $(\wt{J}\times\wt{J}')$-holomorphic Whitney disks in
  $M\times M'$ are transversely cut out. More generally, given $k$-parameter
  families $\wt{J}(t_1,\dots,t_k)$ and $\wt{J}'(t_1,\dots,t_k)$ of eventually
  cylindrical almost complex structures on $M$ and $M'$, $t_i\in[0,1]$, the
  moduli space of holomorphic Whitney disks with respect to the $k$-parameter
  family $(\wt{J}\times\wt{J}')$ is tranvsersally cut out if the moduli spaces
  with respect to $\wt{J}$ and $\wt{J}'$ and transversally cut out and intersect
  transversally in $[0,1]^k$, in which case the moduli space with respect to
  $(\wt{J}\times\wt{J}')$ is the fiber product, over $[0,1]^k$, of the moduli
  spaces with respect to $\wt{J}$ and $\wt{J}'$.
 
  Next, observe that $\ECat(H\times H')=(\ECat H)\times (\ECat H')$.  

  Now, fix sufficiently generic homotopy coherent diagrams
  $F\co \ECat H\to \ol{\JSpace}_M$ and
  $F'\co \ECat H'\to \ol{\JSpace}_{M'}$. Construct a homotopy coherent
  diagram $FF'\co \ECat(H\times H')\to \ol{\JSpace}_{M\times M'}$ as
  follows. On objects, define $FF'(h,h')=F(h)\times F'(h')$. More generally, define
  \begin{multline*}
    FF'((f_n,f'_n),\dots,(f_1,f'_1))(t_1,\dots,t_{n-1})\\
    =[F(f_n,\dots,f_1)(t_1,\dots,t_{n-1})]\times
    [F'(f'_n,\dots,f'_1)(t_1,\dots,t_{n-1})].
  \end{multline*}
  Perturbing $F$ and $F'$ slightly, we may assume that the moduli spaces with respect to
  the family of almost complex structures $F'(f'_n,\dots,f'_1)$ are transverse
  to the moduli spaces with respect to $F(f_n,\dots,f_1)$, so $FF'$ is
  sufficiently generic. Since $F$ and $F'$ were already generic, this perturbation does
  not change the functors $G\co \ECat H\to \Complexes$,
  $G'\co \ECat H'\to \Complexes$.

  Given a
  sequence of morphisms
  $h_0\stackrel{f_1}{\longrightarrow}
  h_1\stackrel{f_2}{\longrightarrow}\cdots\stackrel{f_k}{\longrightarrow}
  h_k$ in $\ECat H$ and a sequence of morphisms
  $h'_0\stackrel{f'_1}{\longrightarrow}
  h'_1\stackrel{f'_2}{\longrightarrow}\cdots\stackrel{f'_\ell}{\longrightarrow}
  h'_\ell$ in $\ECat H'$, a \emph{shuffle} of these sequences is a
  sequence of morphisms
  $(h_0,h'_0)\stackrel{g_0}{\longrightarrow}\cdots\stackrel{g_{k+\ell}}{\longrightarrow}(h_k,h'_\ell)$,
  where each $g_i$ either has the form $(f_j,\Id)$ or
  $(\Id,f'_j)$, and the morphisms $f_1,\dots,f_k$ appear in
  order, once each, in this sequence, as do $f'_1,\dots,f'_\ell$. For
  example, the three shuffles of
  $h_0\stackrel{f_1}{\longrightarrow}h_1\stackrel{f_2}{\longrightarrow}h_2$
  and $h'_0\stackrel{f'_1}{\longrightarrow}h'_1$ are
  \begin{align*}
    &(h_0,h'_0)\stackrel{f_1\times\Id}{\longrightarrow}(h_1,h'_0)\stackrel{f_2\times\Id}{\longrightarrow}(h_2,h'_0)\stackrel{\Id\times f'_1}{\longrightarrow} (h_2,h'_1)\\
    &(h_0,h'_0)\stackrel{f_1\times\Id}{\longrightarrow}(h_1,h'_0)\stackrel{\Id\times f'_1}{\longrightarrow}(h_1,h'_1)\stackrel{f_2\times\Id}{\longrightarrow} (h_2,h'_1)\\
    &(h_0,h'_0)\stackrel{\Id\times f'_1}{\longrightarrow}(h_0,h'_1)\stackrel{f_1\times\Id}{\longrightarrow}(h_1,h'_1)\stackrel{f_2\times\Id}{\longrightarrow} (h_2,h'_1).
  \end{align*}
  The shuffles correspond to permutations $\sigma\in S_{k+\ell}$ so
  that $\sigma|_{\{1,\dots,k\}}$ and $\sigma|_{\{k+1,\dots,k+\ell\}}$
  are increasing.

  Notice that if $(g_1,\dots,g_{k+\ell})$ is a shuffle then the moduli spaces of
  index $-k-\ell+1$ with respect to $FF'(g_{k+\ell},\dots,g_{1})$ are empty unless
  either $k=0$ or $\ell=0$. Indeed, the family
  of almost complex structures $FF'(g_{k+\ell},\dots,g_{1})$
  factors through a map to $[0,1]^{k+\ell-2}$,
  so
  Maslov index $1-k-\ell$ moduli spaces are empty. The exception is if $k=0$
  (respectively $\ell=0$), in which case the moduli space is identified with the
  moduli space of $F'(f'_\ell,\dots,f'_1)$-holomorphic disks (respectively
  $F(f_k,\dots,f_1)$-holomorphic disks), multiplied by constant disks in the
  other factor.

  Let $G\co \ECat H\to \Complexes$, $G'\co \ECat H'\to\Complexes$, and
  $GG'\co\ECat (H\times H')\to\Complexes$ be the homotopy coherent diagrams
  corresponding to $F$, $F'$, and $FF'$, respectively. With notation as
  in~\cite[Definition 3.11]{HEquivariant}, define a map
  \[
    \eta\co (\hocolim G)\otimes (\hocolim G')\to \hocolim GG'
  \]
  by
  \begin{multline*}
    \eta\bigl( (f_k,\dots,f_1;\{0,1\}^{\otimes k};x)\otimes(f'_\ell,\dots,f'_1;\{0,1\}^{\otimes \ell};y) \bigr)
    =\hspace{-2em}\sum_{\text{shuffles }g_1,\dots,g_{k+\ell}}\hspace{-2em}\bigl(g_{k+\ell},\dots,g_1;\{0,1\}^{\otimes k+\ell};x\otimes y\bigr).
  \end{multline*}
  We verify that $\eta$ is a chain map. The terms arising from taking
  the differential of $x$ or $y$ before or after applying $\eta$
  clearly cancel in pairs, so we will ignore them from here on. 

  The remaining terms in $\eta\circ\bdy$ are:
  \begin{align*}
    &\sum_{i=1}^{k}\eta\bigl( (f_k,\dots,f_1;\{0,1\}^{\otimes i-1}\otimes 0\otimes\{0,1\}^{k-i};x)\otimes(f'_\ell,\dots,f'_1;\{0,1\}^{\otimes \ell};y) \bigr)\\
      &\qquad\qquad\qquad+
      \eta\bigl( (f_k,\dots,f_1;\{0,1\}^{\otimes i-1}\otimes 1\otimes\{0,1\}^{k-i}\otimes ;x)\otimes(f'_\ell,\dots,f'_1;\{0,1\}^{\otimes \ell};y) \bigr)\\
    &\qquad+
    \sum_{i=1}^{\ell}\eta\bigl( (f_k,\dots,f_1;\{0,1\}^{\otimes k} ;x)\otimes(f'_\ell,\dots,f'_1;\{0,1\}^{\otimes i-1}\otimes 0\otimes\{0,1\}^{\otimes \ell-i} ;y) \bigr)
    \\
    &\qquad\qquad\qquad+\eta\bigl( (f_k,\dots,f_1;\{0,1\}^{\otimes k} ;x)\otimes(f'_\ell,\dots,f'_1;\{0,1\}^{\otimes i-1}\otimes 1\otimes\{0,1\}^{\otimes \ell-i} ;y) \bigr)\\
    &=
      \sum_{i=1}^{k-1}\eta\bigl( (f_k,\dots,f_{i+1};\{0,1\}^{k-i};G(f_i,\dots,f_1)(\{0,1\}^{i-1}\otimes x)\otimes(f'_\ell,\dots,f'_1;\{0,1\}^{\otimes \ell};y) \bigr)\\
      &\qquad\qquad\qquad+
        \eta\bigl( (f_k,\dots,f_{i+1}\circ f_i,\dots, f_1;\{0,1\}^{\otimes k-1};x)\otimes(f'_\ell,\dots,f'_1;\{0,1\}^{\otimes \ell};y) \bigr)\\
    &\qquad+ \eta\bigl( (f_{k-1},\dots, f_1;\{0,1\}^{\otimes k-1};x)\otimes(f'_\ell,\dots,f'_1;\{0,1\}^{\otimes \ell};y) \bigr)\\
    &\qquad+
    \sum_{i=1}^{\ell}\eta\bigl( (f_k,\dots,f_1;\{0,1\}^{\otimes k} ;x)\otimes(f'_\ell,\dots,f'_{i+1};\{0,1\}^{\otimes \ell-i} ;G'(f'_i,\dots,f'_1)(\{0,1\}^{i-1}\otimes y) \bigr)
    \\
    &\qquad\qquad\qquad+\eta\bigl( (f_k,\dots,f_1;\{0,1\}^{\otimes k} ;x)\otimes (f'_\ell,\dots,f'_{i+1}\circ f'_i,\dots,f'_1;\{0,1\}^{\otimes \ell-1};y) \bigr)\\
    &\qquad+\eta\bigl( (f_k,\dots,f_1;\{0,1\}^{\otimes k} ;x)\otimes (f'_{\ell-1},\dots,f'_1;\{0,1\}^{\otimes \ell-1};y) \bigr).
  \end{align*} 
  Call the terms on the six lines type A1--A6.

  The remaining terms in $\bdy\circ\eta$ are
  \begin{align*}
    \sum_{\text{shuffles }g_1,\dots,g_{k+\ell}}
      &\sum_{i=1}^{k+\ell-1}\bigl(g_{k+\ell},\dots,g_{i+1};\{0,1\}^{k+\ell-i};GG'(g_i,\dots,g_1)(\{0,1\}^{i-1}\otimes (x\otimes y))\bigr)\\
      &\qquad+
        (g_{k+\ell},\dots,g_{i+1}\circ g_i,\dots, g_1;\{0,1\}^{\otimes k+\ell-1};(x\otimes y))\\
    &+(g_{k+\ell-1},\dots, g_1;\{0,1\}^{\otimes k+\ell-1};(x\otimes y)).
  \end{align*}
  Call the terms on the three lines type B1--B3.
  \begin{itemize}
  \item For type B1 terms, if some $g_j$, $1\leq j\leq i$, is of the
    form $f\otimes\Id$ while another $g_{j'}$, $1\leq j'\leq i$, is of
    the form $\Id\otimes f'$ then $GG'(g_i,\dots,g_1)=0$, because the
    corresponding $FF'$-moduli spaces are empty. The remaining type B1
    terms cancel with the type A1 and type A4 terms.
  \item For type B2 terms, if $g_{i+1}=f_j\otimes \Id$ and
    $g_i=\Id\otimes f_{j'}'$ then this term cancels with the
    corresponding term of the shuffle $(g'_1,\dots,g'_{k+\ell})$ which
    agrees with $(g_1,\dots,g_{k+\ell})$ except that
    $g'_{i}=f_{j}\otimes \Id$ and $g'_{i+1}=\Id\otimes f'_{j'}$. The
    remaining type B2 terms cancel with the type A2 and type A5 terms.
  \item For the type B3 term, if $g_{k+\ell}$ is of the form
    $(f_k,\Id)$, then it cancels with the type A3 term, and if
    $g_{k+\ell}$ is of the form $(\Id,f'_\ell)$, then it cancels with
    the type A6 term.
  \end{itemize}
  Thus, $\eta$ is a chain map.

  Clearly, $\eta$ intertwines the actions of $H\times H'$.
  
  For any objects $h\in H$ and $h'\in H'$, the diagram
  \[
    \xymatrix{
      (\hocolim G)\otimes (\hocolim G') \ar[r]^-\eta& \hocolim GG'\\
      G(h)\otimes G(h')\ar[r]_\cong \ar[u]^\simeq &GG'(h\times h')\ar[u]_\simeq
    }
  \]
  commutes.
  Since the bottom horizontal arrow is an isomorphism and the two vertical
  arrows are quasi-isomorphisms, this implies that $\eta$ is a quasi-isomorphism
  as well. This proves the result.
\end{proof}

\begin{corollary}\label{cor:internal-Kunneth}
  Suppose that $H$ acts on both $(M,L_0,L_1)$ and $(M',L'_0,L'_1)$,
  both satisfying~\cite[Hypothesis 3.2]{HEquivariant}. Endowing $(M\times M',L_0\times L_0',L_1\times L_1')$ with the diagonal action of $H$, there is a quasi-isomorphism 
  \[
    \ECF[H](L_0\times L'_0,L_1\times L'_1)\simeq
    \ECF[H](L_0,L_1)\otimes_{\Field} \ECF[H](L'_0,L'_1)
  \]
  as chain complexes over $\Field[H]$ with $H$ acting by the diagonal
  action on the right hand side.
\end{corollary}
\begin{proof}
  With notation as in the proof of Theorem~\ref{thm:external-Kunneth},
  the diagonal map $H\to H\times H$ induces an inclusion map
  $\Delta\co \ECat H\into \ECat H\times \ECat H$. Composing with the
  functor $FF'$ gives a homotopy coherent $\ECat H$-diagram of
  almost complex structures $(FF')\circ\Delta$. The
  corresponding homotopy coherent diagram of chain complexes is
  $(GG')\circ\Delta$. There is an induced map
  \begin{equation}\label{eq:aux-Kunneth-map}
    \hocolim_{\ECat H}[(GG')\circ \Delta]\to \hocolim_{\ECat H\times\ECat H} (GG').
  \end{equation}
  This map clearly respects the $\Field[H]$-module structure, and
  using Theorem~\ref{thm:external-Kunneth}, the two terms are
  quasi-isomorphic to $\ECF[H](L_0\times L'_0,L_1\times L'_1)$ and
  $\ECF[H](L_0,L_1)\otimes_{\Field} \ECF[H](L'_0,L'_1)$ over
  $\Field[H]$. Since for any object $h$ of $\ECat H$ (i.e., element
  $h\in H$) the inclusion of $G(h)\otimes G(h)$ into both
  $\hocolim_{\ECat H\times\ECat H}(GG')$ and
  $\hocolim_{\ECat H}[(GG')\circ \Delta]$ are quasi-isomorphisms, the
  map~\eqref{eq:aux-Kunneth-map} is also a quasi-isomorphism. This
  proves the result.
\end{proof}

\begin{corollary}\label{cor:internal-relaxed-Kunneth}
  Suppose that $H$ acts (symplectically) on symplectic manifolds
  $M,M',N$ and suppose there are $H$-invariant open subsets $V\subset
  M$, $V'\subset M'$ and $U\subset N$ containing $H$-invariant closed
  Lagrangians $L_0,L_1\subset V$, $L'_0,L'_1\subset V'$, and
  $K_0,K_1\subset U$ such that the actions of $H$ on $(M,L_0,L_1)$,
  $(M',L'_0,L'_1)$, $(N,K_0,K_1)$ all satisfy~\cite[Hypothesis
  3.2]{HEquivariant}, and $(U,K_0,K_1)$ is identified
  $H$-symplectically with the product $(V\times V',L_0\times
  L'_0,L_1\times L'_1)$. As in the proofs of
  Theorem~\ref{thm:external-Kunneth} and
  Corollary~\ref{cor:internal-Kunneth}, suppose that there exist
  systems of eventually cylindrical almost complex structures $F$ and
  $F'$ for $M$ and $M'$ and extensions $\wt{(FF')\circ\Delta}$ of
  $(FF')\circ\Delta$ from $U$ to all of $N$, so that $F$, $F'$, and $\wt{(FF')\circ\Delta}$ are
  regular for the freed complexes $\ECF[H](L_0,L_1)$,
  $\ECF[H](L'_0,L'_1)$, and $\ECF[H](K_0,K_1)$ and for which the
  defining holomorphic strips all lie inside $V$, $V'$ and $U$. Then
  there is a quasi-isomorphism
  \[
    \ECF[H](K_0,K_1)\simeq
    \ECF[H](L_0,L_1)\otimes_{\Field} \ECF[H](L'_0,L'_1)
  \]
  as chain complexes over $\Field[H]$ with $H$ acting by the diagonal
  action on the right hand term.
\end{corollary}

\begin{proof}
This follows immediately from Corollary~\ref{cor:internal-Kunneth}.
\end{proof}

\subsection{The K\"unneth theorem for equivariant symplectic Khovanov homology}
In this section we prove an equivariant version of Waldron's K\"unneth
theorem for symplectic Khovanov homology in~\cite[Theorem
1.2]{Waldron:KhSympMaps}.

We will use the following elementary lemma:

\begin{lemma}\label{lem:elementary}
  Let $p(z)$ be a complex polynomial which has simple roots and no
  zeros in the open unit disk $D$. Then there is a smooth 1-parameter
  family of complex polynomials $p_t(z)$, each with no zeros in $D$,
  and only simple roots anywhere,
  interpolating between $p(z)$ and the constant function $1$.
\end{lemma}
\begin{proof}
  The proof is by induction on the degree $n$ of
  $p(z)=a_0+\cdots+a_nz^n$. By multiplying by a path in $\CC^\times$
  from $1$ to $1/a_n$, we may assume $p(z)$ is monic. A monic
  polynomial is uniquely determined by its roots.
  So, there is a path from $p(z)$ to the polynomial
  $(z-2)(z-3)\cdots(z-n-1)$, simply by moving all the roots outside
  the unit disk. Next, let $q(z)=(z-2)(z-3)\cdots(z-n)$ and consider
  the path of polynomials
  \[
    p_t(z)=[(1-t)z-n-1)]q(z).
  \]
  The roots of $p_t(z)$ are $2,3,\dots,n$ and $(n+1)/(1-t)$, all of
  which lie outside the unit circle and are distinct. The polynomial
  $p_1(z)$ has degree $(n-1)$, and all roots outside the unit
  circle. By induction, this completes the proof. (Note that when
  concatenating the paths in the different steps of the proof, one needs
  to reparameterize the paths so the concatenation is smooth.)
\end{proof}

\begin{proposition}\label{prop:disjoint-union}
  Given bridge diagrams $L$ and $L'$, there is a quasi-isomorphism of chain complexes over $\Field[D_{2^m}]$,
  \[
    \eKCSymp(L\amalg L')\simeq \eKCSymp(L)\otimes_{\Field}\eKCSymp(L'),
  \]
  where the right-hand side has the diagonal action of $D_{2^m}$.
\end{proposition}
\begin{proof}
  For symplectic Khovanov homology itself, this is~\cite[Theorem 1.2]{Waldron:KhSympMaps}.

  For the freed Floer complex, we will deduce this result from
  Corollary~\ref{cor:internal-relaxed-Kunneth} after choosing suitable
  open sets $U$, $V$, and $V'$ and deforming the symplectic forms on
  the open set $U$ to be a product (without losing control of the
  holomorphic curves).

  Let $b_1,\dots,b_{2n}$ be the endpoints of the bridges in $L$ (so
  $n$ be the number of bridges in $L$) and $b_{2n+1},\dots,b_{2n+2n'}$
  the endpoints of the bridges in $L'$.  After an isotopy, we may
  assume there are disjoint open disks $U_L,U_{L'}\subset\CC$, containing
  $L$ and $L'$. Let
  \begin{align*}
    S_L&=\{(u,v,z)\in\CC^3\mid u^2+v^2+(z-b_1)\cdots(z-b_{2n})=0\}\\
    S_{L'}&=\{(u,v,z)\in\CC^3\mid u^2+v^2+(z-b_{2n+1})\cdots(z-b_{2n+2n'})=0\}\\
    S_{L\amalg L'}&=\{(u,v,z)\in\CC^3\mid u^2+v^2+(z-b_1)\cdots(z-b_{2n+2n'})=0\}.
  \end{align*}
  Let $\wt{U}_L$ (respectively $\wt{U}_{L'}$) be the preimage of $U_L$ 
  (respectively $U_{L'}$) in $S_{L\amalg L'}$. Let $V_L$ (respectively
  $V_{L'}$) be the preimage of $U_L$ in $S_L$ (respectively the preimage
  of $U_{L'}$  in $S_{L'}$).

  By Point~\ref{item:domain-2} in Section~\ref{sec:background}, any
  holomorphic Whitney disk in $\ssspace{n+n'}$ lies in the subspace
  $U\coloneqq [\Hilb^n(\wt{U}_L)\times \Hilb^{n'}(\wt{U}_{L'})]\cap\ssspace{n+n'}$. If
  we let $\nabla$ denote the subspace of $\Hilb^n(\wt{U}_L)$
  (respectively $\Hilb^{n'}(\wt{U}_{L'})$) where the projection to $\CC$
  has length less than $n$ (respectively $n'$) then
  \[
    [\Hilb^n(\wt{U}_L)\times \Hilb^{n'}(\wt{U}_{L'})]\cap\ssspace{n+n'}=[\Hilb^n(\wt{U}_L)\setminus\nabla]\times[\Hilb^{n'}(\wt{U}_{L'})\setminus\nabla].
  \]
  With respect to the restrictions of the averaged symplectic forms on
  $\Hilb^{n+n'}(S_{L\amalg L'})$ from~\cite[Lemma 4.24]{HLS:Lie} (see
  also Section~\ref{sec:convex}), this identification is a
  symplectomorphism, and it is also biholomorphic with respect to the
  standard complex structures.

  Let
  \begin{align*}
    V&=\Hilb^n(V_L)\setminus\nabla \subset\ssspace{n}\\
    V'&=\Hilb^{n'}(V_{L'})\setminus\nabla \subset\ssspace{n'}.
  \end{align*}
  By the general K\"unneth theorem for the freed Floer complex
  (Corollary~\ref{cor:internal-relaxed-Kunneth}) it suffices to show
  that the freed Floer complex of
  $(\Sigma_{A_1}\times\cdots\times\Sigma_{A_n},\Sigma_{B_1}\times\cdots\times\Sigma_{B_n})$
  inside $\Hilb^n(\wt{U}_L)\setminus\nabla$ is quasi-isomorphic to
  their freed Floer complex inside $V$
  (and similarly for $L'$).
  
  By Lemma~\ref{lem:elementary}, the polynomials
  $p_0(z)=(z-b_1)\cdots(z-b_{2n+2n'})$ and
  $p_1(z)=(z-b_1)\cdots(z-b_{2n})$ can be connected by a smooth family
  of polynomials $p_t(z)$, $t\in[0,1]$ whose roots in $U$ are exactly
  $b_1,\dots,b_{2n}$, and so that all the roots of $p_t(z)$ are
  simple. Let $S_t=\{(u,v,z)\in\CC^3\mid u^2+v^2+p_t(z)=0\}$ and let
  $V_t$ be the preimage of $U$ in $S_t$ ($t\in[0,1]$). The subspaces $V_t$ form a
  smooth family of open complex surfaces. (In particular, each $S_t$
  is smooth, since $p_t$ has only simple roots.)

  Each $V_t$ contains Lagrangian spheres $\Sigma_{A_i}$ and
  $\Sigma_{B_i}$ for $i=1,\dots,n$.  Taking their Hilbert schemes
  gives a smooth family of complex manifolds
  $\Hilb^n(V_t)\setminus\nabla$. Abouzaid-Smith's construction of
  their K\"ahler form $\omega'$ in~\cite[Lemma
  5.5]{AbouzaidSmith:arc-alg} gives a smooth $1$-parameter family of
  K\"ahler forms on $\Hilb^n(V_t)$, which restrict to
  $\Hilb^n(V_t)\setminus\nabla$ as exact forms and which agree with
  the product form outside a neighborhood of the diagonal, so
  $\Sigma_{A_1}\times\cdots\times\Sigma_{A_n}$ and
  $\Sigma_{B_1}\times\cdots\times\Sigma_{B_n}$ are Lagrangian.  The
  averaging construction from~\cite[Lemma 4.24]{HLS:Lie} then gives a
  smooth family of $O(2)$-invariant K\"ahler forms for which
  $\Sigma_{A_1}\times\cdots\times\Sigma_{A_n}$ and
  $\Sigma_{B_1}\times\cdots\times\Sigma_{B_n}$ are still Lagrangian.

  Let $V_t=V_0$ if $t<0$ and $V_t=V_1$ if $t>1$.  The continuation map
  from the freed Floer complex of
  $(\Hilb^n(V_0)\setminus\nabla,
  \Sigma_{A_1}\times\cdots\times\Sigma_{A_n},
  \Sigma_{B_1}\times\cdots\times\Sigma_{B_n})$ to the freed Floer
  complex of
  $(\Hilb^n(V_1)\setminus\nabla,
  \Sigma_{A_1}\times\cdots\times\Sigma_{A_n},
  \Sigma_{B_1}\times\cdots\times\Sigma_{B_n})$ is defined by counting
  $\wt{J}$-holomorphic sections of a bundle $E$ over $\RR\times[0,1]$
  whose fiber over $(t,s)$ is $V_t$, with boundary in the sub-bundle
  $F\subset E$ specified by
  $\Sigma_{A_1}\times\cdots\times\Sigma_{A_n}$ and
  $\Sigma_{B_1}\times\cdots\times\Sigma_{B_n}$ over $\RR\times \{0\}$
  and $\RR\times\{1\}$, respectively. Here, the $\wt{J}$ are suitable
  families of fiberwise almost complex structures $\wt{J}$ satisfying analogues
  of~\ref{item:J-first}--\ref{item:J-last}.
  
  We claim that the manifolds $\Hilb^n(V_t)\setminus\nabla$ are
  uniformly $I$-convex in the following sense, which implies the
  continuation maps are well-defined. Let $j_t$ be the complex
  structure on $V_t$ and let $I_t=\Hilb^n(j_t)$ be the complex
  structure on $\Hilb^n(V_t)$ inherited from $V_t$.  Let $J_t$ be a
  family of almost complex structures satisfying
  conditions~\ref{item:J-first}--\ref{item:J-last} with respect to
  $j_t$ and $I_t$. In particular, assume that $J_t$ agrees with $I_t$
  outside a compact set $K_t$, so that $\bigcup_t K_t$ is also
  compact. The almost complex structures $J_t$ give a fiberwise almost complex
  structure on $E$.  By uniform $I$-convexity we mean there is a
  compact set $K'$, depending only on the $K_t$, so that if $u$ is a
  $J_t$-holomorphic section of $(E,F)$ then the image of $u$ is
  contained in $K'$.
  
  Specifically, fix a family of $I_t$-holomorphic embeddings of the manifolds
  $\Hilb^n(S_t)\setminus\nabla$ in $\CC^N$ for some large $N$. Then
  $K'$ is the intersection of $\bigcup_t\Hilb^n(V_t)\setminus\nabla$ with:
  \begin{itemize}
  \item $\Hilb^n(V'_t)$ where $V'_t$ is the preimage of a slightly
    smaller open set $U'$ with $\overline{U'}\subset U$, and
  \item the polydisk $\{(z_1,\dots,z_n)\in\CC^N\mid |z_i|\leq R\}$
    where $R$ is large enough that this set contains the Lagrangians
    and the compact sets $K_t$.
  \end{itemize}
  The fact that the image of a holomorphic curve $u$ lies in $K'$
  follows from positivity of intersections (for the first term in the
  intersection, as in~\ref{item:domain-2} above) and the maximum
  modulus theorem (for the second term in the intersection).

  %

  As noted above, this convexity is enough to ensure that the proof of
  invariance of the freed Floer complex~\cite[Proposition
  3.28]{HEquivariant} applies. Hence, the freed Floer complexes in
  $\Hilb^n(V_0)\setminus\nabla$ and $\Hilb^n(V_1)\setminus\nabla$
  agree up to quasi-isomorphism, completing the proof.
\end{proof}

\subsection{The basepoint action}
The last ingredient in the proof of equivariant stabilization
invariance is a module structure on symplectic Khovanov homology,
analogous to one on Khovanov homology. Fix a bridge diagram $L$ and a
basepoint $p\in L$ on one the $A$-arcs, say $A_i$. There is an
action of $\Field[X]$ on the symplectic Khovanov complex of $L$
defined as follows. Choose a preimage $p_e\in \Sigma_{A_i}$ lying over
$p\in A_i$. For any $q\in \Sigma_{A_i}$, let
$\mathcal{O}_q\subset \CipLag_A$ denote the codimension-2 subspace
where one of the coordinates is $q$. Then the action of $X$ counts
rigid holomorphic strips $u\from \RR\times[0,1]\to \ssspace{n}$ with
$u(0,0)\in \mathcal{O}_{p_e}$. (The fact that such moduli spaces with
a point constraint are transversely cut out is a straightforward
adaptation of~\cite[Theorem
3.4.1]{MS04:HolomorphicCurvesSymplecticTopology} to the relative
case.)

\begin{theorem}\label{thm:unknot-action-2}
  The above action of $\Field[X]$ on the symplectic
  Khovanov complex of $L$ satisfies the following properties:
  \begin{enumerate}[label=(BP-\arabic*)]
  \item\label{item:bp-square} Multiplication by $X^2$ is homotopic to
    $0$, so symplectic Khovanov homology inherits an action of
    $\Field[X]/(X^2)$.
  \item\label{item:bp-commute} If $p,p'\in L$ are different points then
    multiplication by $X$ at $p$ and at $p'$ commute up to homotopy.
  \item\label{item:bp-unknot} The symplectic Khovanov homology of the
    $1$-bridge unknot is isomorphic to $\Field[X]/(X^2)$.
  \item\label{item:bp-invariance} Up to homotopy, the chain maps
    associated to changes of almost complex structures on
    $\ssspace{n}$, isotopies and handleslides of the $A$- and $B$-arcs
    which do not move the point $p$, and diffeomorphisms (which may
    move $p$), commute with multiplication by $X$.
  \item\label{item:bp-skein} The maps in the skein exact triangle from
    Theorem~\ref{thm:skein-tri} respect the action by
    $\Field[X]/(X^2)$ (where corresponding points $p$ are used for the
    three diagrams).
  \item\label{item:bp-move-bp} If $A_i,B_i,A_{i+1}$ are adjacent arcs
    in a bridge diagram so that the interior of $B_i$ does not
    intersect any $A$-arc and the interior of $A_i$ does not intersect
    any $B$-arc---that is, if the configuration looks like
    Figure~\ref{fig:stabilization}(b)---and if $p\in A_i$ and $p'\in
    A_{i+1}$, then the actions by $p$ and $p'$ are homotopic.
  \item\label{item:bp-Kunneth} If $L$ is a disjoint union of two
    bridge diagrams $L_1\amalg L_2$ and $p$ is a basepoint on $L_1$
    then the K\"unneth theorem
    $\KCSymp(L)\simeq \KCSymp(L_1)\otimes_\Field \KCSymp(L_2)$
    (from~\cite[Theorem 1.2]{Waldron:KhSympMaps} or the
    non-equivariant version of Proposition~\ref{prop:disjoint-union})
    intertwines the action of $X$ on $\KCSymp(L)$ and the action of
    $X$ on $\KCSymp(L_1)\otimes_\Field \KCSymp(L_2)$ induced from the
    action of $X$ on $\KCSymp(L_1)$. That is, for appropriate choices
    of almost complex structures, $(m\otimes n)*_pX=(m*_pX)\otimes n$.
  \end{enumerate}
\end{theorem}
\begin{proof}
  Some of these properties (namely \ref{item:bp-square},
  \ref{item:bp-commute}, \ref{item:bp-skein}, and a part of
  \ref{item:bp-invariance}) are fairly standard (see,
  e.g.,~\cite[Section 8l]{SeidelBook} or \cite[Section
  3.9]{Perutz:Matching2}, and references therein) and hold for
  Lagrangian Floer homology more generally in the absence of disk bubbles, so
  we will only sketch the proofs of those properties and concentrate
  on the properties that are specific to symplectic Khovanov
  homology.

  For Property~\ref{item:bp-square}, let $p'_e$ be a point in the
  fiber over $p$ close to $p_e$. For a generic choice of almost
  complex structure, if we choose $p'_e$ close enough to $p_e$ then
  the actions induced by $p_e$ and $p'_e$ agree. Observe that
  $O_{p_e}\cap O_{p'_e}=\emptyset$. Counting holomorphic bigons with
  $u(0,0)\in O_{p_e}$ and $u(t,0)\in O_{p'_e}$ for some $t>0$ gives a
  homotopy from multiplication by $X^2$ (corresponding to
  $t\to\infty$) to $0$ (corresponding to $t=0$).

  For Property~\ref{item:bp-commute}, let $*_pX$ and $*_{p'}X$ be the
  actions at $p$ and $p'$, respectively. Counting holomorphic disks
  with $u(0,0)\in O_{p_e}$ and $u(0,t)\in O_{p'_e}$, for any
  $t\in\RR$, gives a homotopy between $*_pX*_{p'}X$ (for
  $t\to -\infty$) and $*_{p'}X*_pX$ (for $t\to +\infty$).
  
  For Property~\ref{item:bp-unknot}, observe that, with respect to the
  standard complex structure on $S$, there is an $S^1$-family of
  holomorphic disks connecting the two generators of $\KCSymp(U)$, one
  through each preimage of $p$ on $\Sigma_A=\CipLag_A$. Indeed, the
  subset $S_0=\{(u,v,z)\in S\mid v=0\}$ is bi-holomorphic to the
  cylinder $\RR\times S^1$, and $\Sigma_A\cap S_0$ and
  $\Sigma_B\cap S_0$ are circles inside $S_0$ intersecting in two
  points. There are two holomorphic disks in $S_0$, transversally cut
  out in $S_0$, whose images under $S^1$ form a single family of
  holomorphic disks passing through each preimage of $p$ on
  $\Sigma_A$. It follows from a doubling argument and automatic
  transversality as in~\cite{HLS97:GenericityHoloCurves} or,
  equivalently, a doubling argument and~\cite[Lemma
  3.3.1]{MS04:HolomorphicCurvesSymplecticTopology}, that these
  holomorphic disks are transversally cut out in $S$. Next, consider
  the involution $\tau\co S\to S$, $\tau(u,v,z)=(u,-v,z)$. Since the
  holomorphic disks inside $\Fix(\tau)=S_0$ are transversally cut out,
  we can perturb the complex structure slightly to a $\tau$-invariant
  almost complex structure in which all holomorphic bigons are
  transversally cut out (see, e.g.,~\cite[Section
  5c]{KhS02:BraidGpAction}).  Choose the preimage $p_e$ of $p$ to lie
  in $S_0$. Then any holomorphic bigons not contained in $S_0$ passing
  through $p_e$ come in pairs exchanged by $\tau$, and hence
  contribute $0 \pmod{2}$ to the disk count.
  
  Property~\ref{item:bp-invariance} is proved by considering
  holomorphic disks with cylindrical-at-infinity complex structures
  (for changes of complex structures or isotopies of the $A$- or
  $B$-arcs) or holomorphic triangles (for handleslides) with a similar
  point constraint.  This in particular includes isotopies of $B$-arcs
  that pass over $p$. For the case of a handleslide between $A$-arcs,
  there is an additional complication, so we spell out that case. Let
  $A'$ be a collection of bridges obtained from the $A$ bridges by a
  handleslide, arranged in the plane so that the $A$- and $A'$-bridges
  intersect only at their endpoints. Fix points $p\in A_i$ and
  $p'\in A'_i$, where $A'_i$ is the $A'$-arc corresponding to $A_i$, which is either a small
  translate of $A_i$, or, if $A_i$ is being handleslid, the result of
  the handleslide. Let $p_e\in\Sigma_{A_i}$ over $p$, and
  $p'_e\in\Sigma_{A'_i}$ over $p'$. We will focus on the case that $p$
  is on the arc being handleslid; the other cases are similar but easier.

  Consider the $1$-dimensional moduli space of holomorphic triangles
  with boundary on $(\CipLag_A,\CipLag_{A'},\CipLag_{B})$ with either a
  point along the $A$-edge mapped to $\mathcal{O}_{p_e}$ or a
  point along the $A'$-edge mapped to $\mathcal{O}_{p'_e}$, and the
  corner at $(\CipLag_A,\CipLag_{A'})$ mapped to the generator
  $1$. This moduli space has six kinds of ends:
  \begin{enumerate}[label=(TM-\arabic*)]
  \item Ends corresponding to a bigon for $(\CipLag_A,\CipLag_B)$ with
    a point mapping to $\mathcal{O}_{p_e}$ and a holomorphic triangle. These
    correspond to following the basepoint action for
    $(\CipLag_A,\CipLag_B)$ by the isomorphism
    $F\co \HF(\CipLag_A,\CipLag_B)\to\HF(\CipLag_{A'},\CipLag_{B})$.
  \item Ends corresponding to a bigon for $(\CipLag_{A'},\CipLag_{B})$
    with a point mapping to $\mathcal{O}_{p'_e}$ and a holomorphic triangle. These
    correspond to following $F$ by the basepoint action for
    $(\CipLag_{A'},\CipLag_B)$.
  \item Ends corresponding to a bigon for $(\CipLag_A,\CipLag_B)$ and
    a holomorphic triangle with a point constraint. These correspond
    to $H\circ\bdy$, where $H$ is a homotopy defined by counting rigid
    triangles with a point constraint.
  \item Ends corresponding to a bigon for $(\CipLag_{A'},\CipLag_{B})$
    and a holomorphic triangle with a point constraint. These
    correspond to $\bdy\circ H$, where $H$ is a homotopy defined by
    counting rigid triangles with a point constraint.
  \item\label{item:TM-5} Ends corresponding to a bigon for
    $(\CipLag_A,\CipLag_{A'})$ with a point mapping to
    $\mathcal{O}_{p_e}$ and a holomorphic triangle.
  \item\label{item:TM-6} Ends corresponding to a bigon for
    $(\CipLag_A,\CipLag_{A'})$ with a point mapping to
    $\mathcal{O}_{p'_e}$ and a holomorphic triangle.
  \end{enumerate}
  We need to show that the last two cases cancel. The triangles in the last two
  cases are the same, so we need to know that the number of bigons in the two
  cases agree (modulo 2). This count of bigons is exactly the basepoint action
  for the bridge diagram $(A,A')$, with either a basepoint on $A_i$ or
  $A'_i$. Since the differential on $\CF(\CipLag_A,\CipLag_{A'})$ is trivial for
  grading reasons, it suffices to verify that the two basepoint actions are
  homotopic.

  \begin{figure}
    \centering
    \includegraphics{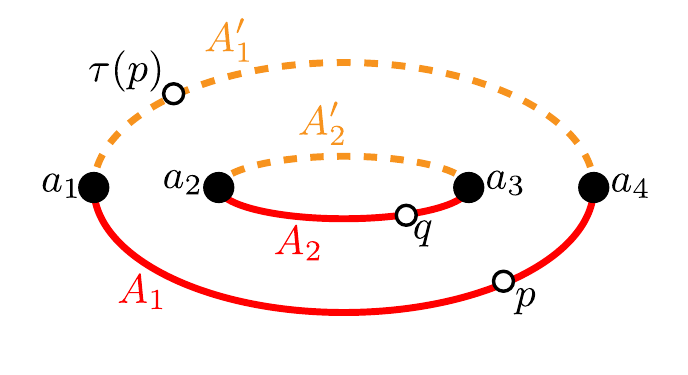}
    \caption{\textbf{Notation for proof of handleslide invariance.}
      The arcs $A_i$ and $A'_i$, their endpoints $a_1,\dots,a_4$, and
      the basepoints $p$, $\tau(p)$, and $q$ are shown.}
    \label{fig:handleslide-notation}
  \end{figure}
  
  By the proof of Proposition~\ref{prop:disjoint-union} (in the
  slightly simpler, non-equivariant case), it suffices to consider the
  case that $A$ consists of only the two arcs involved in the
  handleslide. Label the endpoints of these arcs $a_1,\dots,a_4$, so
  that $a_1$ and $a_4$ are on the outer circle of $A\cup A'$. (See
  Figure~\ref{fig:handleslide-notation}.) The generators of
  $\CF(\CipLag_A,\CipLag_{A'})$ are $\{a_1,a_2\}$, $\{a_1,a_3\}$,
  $\{a_4,a_2\}$, and $\{a_4,a_3\}$. Label the arcs so that
  $\{a_1,a_2\}$ is the top-graded generator, denoted $1$ above. Let $A_1$ be the arc with
  $\bdy A_1=\{a_1,a_4\}$ and $A_2$ the arc with
  $\bdy A_2=\{a_2,a_3\}$.

  Let $\tau\co \CC\to\CC$ be rotation by $\pi$. Arrange that $\tau$ exchanges
  $A$ and $A'$, and in particular exchanges $A_i$ and $A'_i$. There is an induced
  map $\tau\co \ssspace{n}\to\ssspace{n}$ which exchanges $\CipLag_A$ and
  $\CipLag_{A'}$.

  The basepoint $p$ lies on $A_1$. There is a corresponding basepoint
  $\tau(p)$ on $A'_1$. Choose also a basepoint $q$ on $A_2$. We claim
  that:
  \begin{enumerate}
  \item $\{a_1,a_2\}*_q X=\{a_1,a_3\}$ and $\{a_4,a_2\}*_q X=\{a_4,a_3\}$.
  \item The coefficient of $\{a_1,a_3\}$ in $\{a_1,a_2\}*_p X$ is 0,
    as is the coefficient of $\{a_1,a_3\}$ in
    $\{a_1,a_2\}*_{\tau(p)}X$.
  \item The  coefficient of $\{a_4,a_2\}$ in $\{a_1,a_2\}*_{\tau(p)}X$ is the
    same as the coefficient of $\{a_4,a_3\}$ in $\{a_1,a_3\}*_pX$.
  \end{enumerate}
  Together, these three claims imply the result. Indeed, our goal is to show that
  $\{a_1,a_2\}*_pX=\{a_1,a_2\}*_{\tau(p)}X$.  By the second claim,
  $\{a_1,a_2\}*_pX=\epsilon\{a_4,a_2\}$ and
  $\{a_1,a_2\}*_{\tau(p)}X=\delta\{a_4,a_2\}$ for some
  $\epsilon,\delta\in \FF_2$; our goal is to show that
  $\epsilon=\delta$.  By Property~\ref{item:bp-commute},
  \[
  \{a_1,a_2\}*_q X*_pX=\{a_1,a_2\}*_pX*_qX.
  \]
  For the left hand side, by the first claim, $\{a_1,a_2\}*_q
  X=\{a_1,a_3\}$, and by the third claim,
  $\{a_1,a_3\}*_pX=\delta\{a_4,a_3\}$. For the right hand side,
  $\{a_1,a_2\}*_pX=\epsilon\{a_4,a_2\}$, and by the first claim,
  $\epsilon\{a_4,a_2\}*_q X=\epsilon\{a_4,a_3\}$. Therefore, $\epsilon=\delta$.
  
  It remains to verify the three properties.  Recall from
  Section~\ref{sec:background} that a Whitney disk in $\ssspace{n}$
  has a projected domain.
  For the first claim, observe
  that the projected domain of a holomorphic bigon counted for
  $\{a_1,a_2\}*_q X$ is contained entirely in the bounded region of
  $\CC\setminus(A_2\cup A'_2)$. Thus, the proof of the K\"unneth
  theorem shows that this computation is the same as in the $1$-bridge
  unknot case, so this follows from
  Property~\ref{item:bp-unknot}. Similarly, the second claim
  follows from the fact that there is no projected domain compatible
  with this action.

  For the last claim we use the involution $\tau$. Given a
  $J$-holomorphic Whitney disk $u\co \RR\times[0,1]\to \ssspace{n}$
  for $(\CipLag_A,\CipLag_{A'})$ connecting $\{a_1,a_2\}$ to
  $\{a_4,a_2\}$ (so $u(\RR\times\{0\})\subset\CipLag_A$ and
  $\lim_{s\to-\infty}u(s,t)=\{a_1,a_2\}$), there is a corresponding
  $(\tau_*J)$-holomorphic disk
  $\tau\circ u\co \RR\times[0,1]\to \ssspace{n}$ for
  $(\CipLag_{A'},\CipLag_A)$ connecting $\{a_4,a_3\}$ to $\{a_1,a_3\}$
  (that is, $(\tau\circ u)(\RR\times\{0\})\subset\CipLag_{A'}$ and
  $\lim_{s\to -\infty}(\tau\circ u)(s,t)=\{a_4,a_3\}$). There is a
  holomorphic map $\sigma\co \RR\times[0,1]\to\RR\times[0,1]$ given by
  $\sigma(s,t)=(1-s,-t)$ (rotation around the middle of the
  strip). Then $\tau\circ u\circ \sigma$ is a $(\tau_*J)$-holomorphic
  Whitney disk for $(\CipLag_A,\CipLag_{A'})$ connecting $\{a_1,a_3\}$
  to $\{a_4,a_3\}$. Further, $u$ passes through $\mathcal{O}_{\tau(p_e)}$ if
  and only if $\tau\circ u\circ\sigma$ passes through
  $\mathcal{O}_{p_e}$. So, if we take $p'_e=\tau(p_e)$ then the
  basepoint action using $p'_e$ and the almost complex structure $J$
  agrees with the basepoint action using $p_e$ and the almost complex
  structure $\tau_*(J)$. Since the basepoint action on homology is
  independent of the choice of almost complex structure and the
  differential on the Floer complex is trivial, it follows
  that the coefficient of $\{a_4,a_2\}$ in $\{a_1,a_2\}*_{\tau(p)}X$ is the
  same as the coefficient of $\{a_4,a_3\}$ in
  $\{a_1,a_3\}*_{p}X$, as desired. This concludes the proof of
  Property~\ref{item:bp-invariance}.
  
  Property~\ref{item:bp-skein} follows from the same argument as the case of
  $B$-handleslides in Property~\ref{item:bp-invariance}, by considering moduli
  spaces of holomorphic triangles with a point constraint. (Because only one
  edge of the triangle has a point constraint for the skein maps or a
  $B$-handleslide, there are no degenerations analogous to types~\ref{item:TM-5}
  and~\ref{item:TM-6}, making these cases easier than the case of
  $A$-handleslides.)

  \tikzstyle{aarc}=[draw={rgb,255:red,63;green,64;blue,150},line width=1pt,solid,-]
  \tikzstyle{barc}=[draw={rgb,255:red,238;green,51;blue,56},line width=1pt,dash pattern={on 2pt off 1pt},-]  
  \tikzstyle{inter}=[anchor=center,circle,inner sep=0,outer sep=0,minimum width=5pt,fill={rgb,255:red,55;green,53;blue,53}]
  \tikzstyle{outerr}=[outer sep=10pt,draw={rgb,255:red,55;green,53;blue,53},dash pattern={on 1pt off 1pt},line width=0.5pt,rounded corners]
  \begin{figure}
    \begin{tikzpicture}[xscale=1.35]
      \node[outerr] (initial) at (0,0) {
        \begin{tikzpicture}[xscale=-1,yscale=-1]
          \draw[aarc] (-0.5,-0.5) to[out=0,in=-90] (0,0);
          \draw[barc] (0,0) -- (1,0);
          \draw[aarc] (1,0) -- (2,0);
          \draw[barc] (2,0) -- (2,1);
          \draw[aarc] (2,1) -- (0,1) to[out=180,in=-90] (-0.5,1.5);
          
          \draw[aarc] (-0.5,0.4) -- (2.5,0.4);
          \draw[aarc] (-0.5,0.6) -- (2.5,0.6);
          
          \node[inter] at (0,0) {};
          \node[inter] at (1,0) {};
          \node[inter] at (2,0) {};
          \node[inter] at (2,1) {};
        \end{tikzpicture}};
      \draw[->] (initial.east) --++(1,0) node[pos=0.5,anchor=south] {\tiny handleslides}
      node[pos=1,anchor=west,outerr] (handleslid) {
        \begin{tikzpicture}[xscale=-1,yscale=-1]
          \draw[aarc] (-0.5,-0.5) to[out=0,in=-90] (0,0);
          \draw[barc] (0,0) -- (1,0);
          \draw[aarc] (1,0) -- (2,0);
          \draw[barc] (2,0) -- (2,1);
          \draw[aarc] (2,1) -- (0,1) to[out=180,in=-90] (-0.5,1.5);
          
          \draw[aarc,rounded corners] (-0.5,0.4) -- (0.6,0.4) -- (0.6,-0.4)-- (2.4,-0.4) -- (2.4,0.4) -- (2.5,0.4); 
          \draw[aarc,rounded corners] (-0.5,0.6) -- (0.8,0.6) -- (0.8,-0.2)-- (2.2,-0.2) -- (2.2,0.6) -- (2.5,0.6); 
          
          \node[inter] at (0,0) {};
          \node[inter] at (1,0) {};
          \node[inter] at (2,0) {};
          \node[inter] at (2,1) {};
        \end{tikzpicture}};
      \draw[->] (handleslid.east) --++(1,0) node[pos=0.5,anchor=south] {\tiny diffeomorphism}
      node[pos=1,anchor=west,outerr] (isotoped) {
        \begin{tikzpicture}[xscale=-1,yscale=-1]
          \draw[aarc] (-0.5,-0.5) to[out=0,in=-90] (0,0);
          \draw[barc] (0,1) -- (1,1);
          \draw[aarc] (1,1) -- (2,1);
          \draw[barc,rounded corners] (0,0) -- (2,0) -- (2,1);
          \draw[aarc] (0,1) to[out=180,in=-90] (-0.5,1.5);
          
          \draw[aarc] (-0.5,0.4) -- (2.5,0.4);
          \draw[aarc] (-0.5,0.6) -- (2.5,0.6);
          
          \node[inter] at (0,0) {};
          \node[inter] at (1,1) {};
          \node[inter] at (2,1) {};
          \node[inter] at (0,1) {};
        \end{tikzpicture}};
    \end{tikzpicture}
    \caption{\textbf{Moving the basepoint.}  Using handleslides and
      diffeomorphism, one can move a small pair (a pair of adjacent $A$ and
      $B$ arcs whose interiors are disjoint from all bridges) across
      the next arc. For the configuration from
      Figure~\ref{fig:stabilization}(b), the basepoint can be moved
      from the $A$-arc in the small pair to the adjacent $A$-arc by
      performing this local move once around the entire link
      component.}\label{fig:move-basepoint}
  \end{figure}
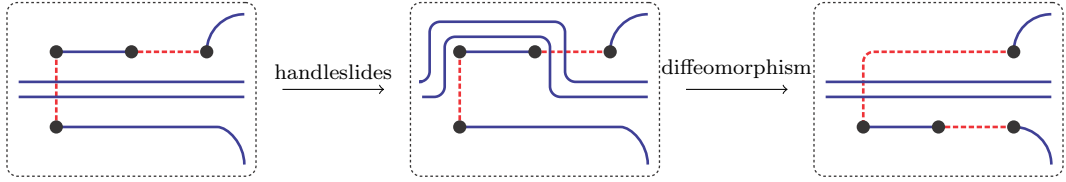

  For Property~\ref{item:bp-move-bp}, we use the well-known trick of moving the
  basepoint using handleslides and diffeomorphisms. Call a pair of adjacent
  $A$ and $B$ arcs \emph{a small pair} if their interiors are disjoint
  from all bridges, for example as in
  Figure~\ref{fig:stabilization}(b). Figure~\ref{fig:move-basepoint}
  shows how a small pair can be moved across an adjacent arc using
  handleslides and a diffeomorphism. Repeating this move
  once around the link component has the effect of moving the
  basepoint from the $A$-arc in a small pair to the next $A$-arc.

  Property~\ref{item:bp-Kunneth} is immediate from the definitions and
  the proof of Proposition~\ref{prop:disjoint-union}.
\end{proof}

\subsection{Proof of equivariant stabilization invariance}

\begin{figure}
  \centering
  \begin{overpic}[tics=10]{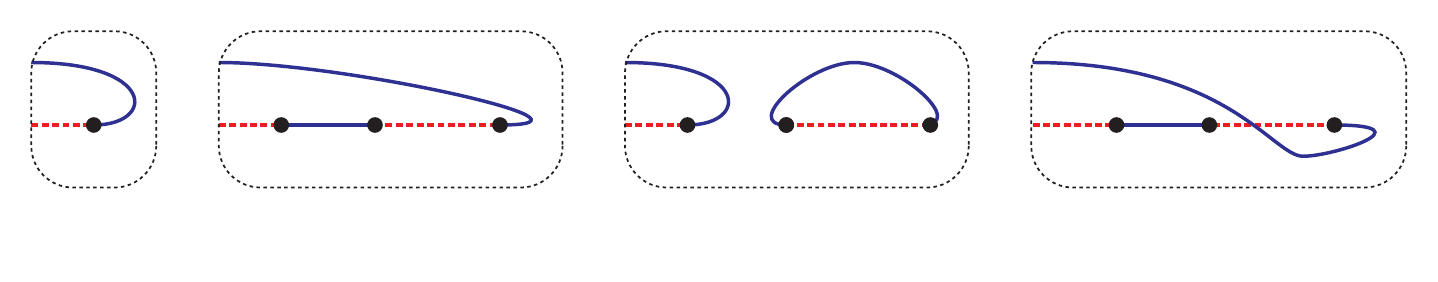}
    \put(5,2){(a)}
    \put(25,2){(b)}
    \put(54,2){(c)}
    \put(83,2){(d)}
  \end{overpic}
  \caption{\textbf{Stabilization and the skein sequence.} This
    is~\cite[Figure 3]{HEquivariant}. (a) A piece of a bridge
    diagram. $A_i$ is \textcolor{red}{dashed} and $B_i$ is
    \textcolor{blue}{solid}. (b) A stabilization of the bridge
    diagram. (c) The disjoint union of the bridge diagram with an
    unknot. (d) A Reidemeister I move applied to the
    stabilization. Diagrams (b) and (c) are obtained from (d) by
    unoriented skein moves.}
  \label{fig:stabilization}
\end{figure}
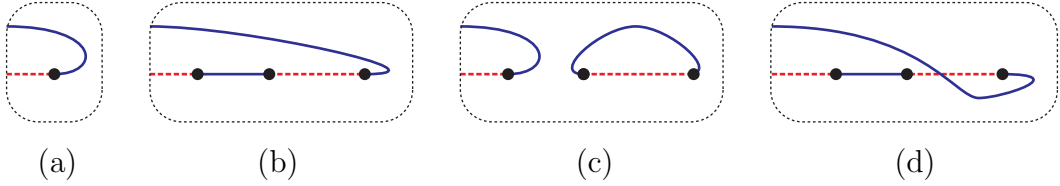

\begin{proposition}\label{prop:equi-stab-inv}
  Let $(\{A_i\},\{B_i\})$ and $(\{A'_i\},\{B'_i\})$ be bridge diagrams
  for a link $K$ which differ by a single stabilization, as in
  Figure~\ref{fig:stabilization}(a,b). Then there is a quasi-isomorphism
  $\eKCSymp(\{A_i\},\{B_i\})\simeq \eKCSymp(\{A'_i\},\{B'_i\})$ over $\Field[D_{2^m}]$.
\end{proposition}
\begin{proof}
  Let $L$ be a link diagram as in Figure~\ref{fig:stabilization}(a),
  $L_1$ as in Figure~\ref{fig:stabilization}(b), $L_0$ as in
  Figure~\ref{fig:stabilization}(c), and $L'$ as in
  Figure~\ref{fig:stabilization}(d). Our goal is to show that there is
  a quasi-isomorphism
  \[
    \eKCSymp(L)\simeq\eKCSymp(L_1)
  \]
  over $\Field[D_{2^m}]$.

  Let $p$ be a point in $L$ on the $A$-arc in the figure and $q$ a point
  on $A_{n+1}$.  By Proposition~\ref{prop:disjoint-union}, there is a
  quasi-isomorphism
  \[
    \eKCSymp(L_0)\simeq \eKCSymp(L)\otimes\eKCSymp(U)
  \]
  over $\Field[D_{2^m}]$.  On homology, this gives an isomorphism
  \[
    \KhSymp(L_0)\cong \KhSymp(L)\otimes\KhSymp(U).
  \]
  Theorem~\ref{thm:unknot-action-2} gives two actions of $\Field[X]/(X^2)$, one coming from the point $p$
  and one from the point $q$, and this isomorphism respects the
  actions.  Write the action at $p$ (respectively $q$) as $*_p$
  (respectively $*_q$). We claim that
  \begin{equation}\label{eq:p-is-q-subset}
    \KhSymp(L)\cong\{m\in\KhSymp(L)\otimes\KhSymp(U)\mid m*_pX=m*_qX\}.
  \end{equation}
  Indeed, from Properties~\ref{item:bp-unknot} and~\ref{item:bp-Kunneth}, $a\otimes 1+b\otimes X$ is an element of the right-hand side
  if and only if $a*_pX=0$ and $a=b*_pX$, but the second equation
  implies the first. So, this set is exactly
  $\{(b*_pX\otimes 1)+(b\otimes X)\}$, which is isomorphic (as an
  $\FF_2$-vector space) to $\KhSymp(L)$.

  Now, consider the skein exact triangle
  \[
    \cdots\to \KhSymp(L')\to \KhSymp(L_0)\to \KhSymp(L_1)\to\cdots.
  \]
  By Theorem~\ref{thm:SS-invt} and (the slightly simpler,
  non-equivariant case of) Proposition~\ref{prop:disjoint-union},
  this triangle is, in fact, a short exact sequence
  \[
    0\to \KhSymp(L')\to \KhSymp(L_0)\to \KhSymp(L_1)\to 0.
  \]
  Let $g_*\co \KhSymp(L_0)\to \KhSymp(L_1)$ be the map from this
  sequence. Since Property~\ref{item:bp-move-bp} implies the actions at $p$ and $q$ are the same on
  $\KhSymp(L_1)$, by Property~\ref{item:bp-skein}
  the map $g_*$ sends the image of $(*_pX-*_qX)$ to $0$.
  Hence, by exactness and comparing dimensions, $g_*$ sends
  \[
    \KhSymp(L_0)/\{m\in\KhSymp(L)\otimes\KhSymp(U)\mid m*_pX=m*_qX\}\cong \KhSymp(L)
  \]
  isomorphically to $\KhSymp(L_1)$.
  There is a chain map over $\Field[D_{2^m}]$
  \[
    f\co \eKCSymp(L)\to \eKCSymp(L)\otimes_\Field\eKCSymp(U)
  \]
  induced by 
  $f(y)=y\otimes 1$.  The induced map on homology is an isomorphism
  from $\KhSymp(L)$ to
  $\KhSymp(L)/\{m\in\KhSymp(L)\otimes\KhSymp(U)\mid m*_pX=m*_qX\}$.
  
  Since the map $g_*$ is induced by counting holomorphic triangles
  with one corner at a $D_{2^m}$-invariant intersection point, the map
  $g_*$ is induced by a map $g\co \eKCSymp(L_0)\to \eKCSymp(L_1)$
  (see, e.g.,~\cite[Proof of Proposition 3.25]{HEquivariant}). The
  composition $g\circ f\co \eKCSymp(L)\to \eKCSymp(L_1)$ is the
  desired quasi-isomorphism.
\end{proof}

\begin{proof}[Proof of Theorem 1.26]
  This is immediate from isotopy and handleslide invariance, verified
  in our original proof of the theorem, and
  Proposition~\ref{prop:equi-stab-inv}.
\end{proof}

\bibliographystyle{hamsalpha}
\bibliography{heegaardfloer}
\end{document}
